\documentclass[10pt, a4paper, reqno]{amsart}
\usepackage{amscd}
\usepackage{dsfont}

\usepackage{amsmath}
\usepackage{amstext}
\usepackage{amsthm}
\usepackage{amssymb}
\usepackage{amsopn}
\usepackage{amssymb}
\usepackage{amsxtra}
\usepackage{amsfonts}
\usepackage{fancyhdr}
\usepackage{paralist, epsfig}
\usepackage[mathcal]{euscript}
\usepackage[english]{babel}
\usepackage[latin1]{inputenc}

\usepackage{todonotes}

\makeatletter
\g@addto@macro\normalsize{%
  \setlength\abovedisplayskip{10pt}
  \setlength\belowdisplayskip{10pt}
  \setlength\abovedisplayshortskip{5pt}
  \setlength\belowdisplayshortskip{8pt}
}

\allowdisplaybreaks

\newtheoremstyle{normal}
{5pt}
{5pt}
{\normalfont}
{}
{\bfseries}
{}
{0.4em}
{\bfseries{\thmname{#1}\thmnumber{ #2}.\thmnote{ \hspace{0.5em}(#3)\newline}}}

\newtheoremstyle{kursiv}
{5pt}
{5pt}
{\itshape}
{}
{\bfseries}
{}
{0.4em}
{\bfseries{\thmname{#1}\thmnumber{ #2}.\thmnote{ \hspace{0.5em}(#3)\newline}}}

\theoremstyle{kursiv}

\theoremstyle{normal}

\newtheorem{thm}{Theorem}
\newtheorem{ex}[thm]{Example}
\newtheorem{rem}[thm]{Remark}
\newtheorem{cor}[thm]{Corollary}
\newtheorem{lem}[thm]{Lemma}
\newtheorem{prop}[thm]{Proposition}

\renewcommand{\epsilon}{\varepsilon}

\newcommand{\T}{(T(t))_{t\geqslant0}}

\newcommand{\id}{\operatorname{id}\nolimits}

\newcommand{\cs}{\operatorname{cs}\nolimits}

\newcommand{\Bigsum}[2]{\ensuremath{\mathop{\textstyle\sum}_{#1}^{#2}}}

\newcommand{\Cnull}{\operatorname{C_0}\nolimits}

\usepackage{setspace}
\usepackage{color}

\definecolor{grey}{gray}{.3}

\newcommand{\N}{\mathbb{N}}
\newcommand{\norm}[1]{\left\lVert #1 \right\rVert}
\newcommand{\abs}[1]{\left | #1 \right |}

\begin{document}

\title{Non Power bounded generators\\of strongly continuous semigroups}

\author{Anna Goli\'{n}ska\hspace{0.5pt}\MakeLowercase{$^{\text{1,\hspace{1pt}a}}$} and Sven-Ake Wegner\hspace{0.5pt}\MakeLowercase{$^{\text{2}}$}}

\renewcommand{\thefootnote}{}
\hspace{-1000pt}\footnote{\hspace{5.5pt}2010 \emph{Mathematics Subject Classification}: Primary 47D06; Secondary 34G10, 46E10, 46A03, 46A45.}
\hspace{-1000pt}\footnote{\hspace{5.5pt}\emph{Key words and phrases}: Strongly continuous semigroup, power bounded operator, Fr\'{e}chet space\vspace{1.6pt}}

\hspace{-1000pt}\footnote{\hspace{0pt}$^{1}$\,Faculty of Mathematics and Computer Science, Adam Mickiewicz University, Umultowska 87, 61-614 Pozna\'{n}, Poland,\linebreak\phantom{x}\hspace{1.2pt}Phone: +48\hspace{1.2pt}61\hspace{1.2pt}829\hspace{1.2pt}-\hspace{1.2pt}54\hspace{1.2pt}-\hspace{1.2pt}80, E-Mail: czyzak@amu.edu.pl.\vspace{1.6pt}}

\hspace{-1000pt}\footnote{\hspace{0pt}$^{2}$\,Corresponding author: Bergische Universit\"at Wuppertal, FB C -- Mathematik, Gau\ss{}stra\ss{}e 20, 42119 Wuppertal, Ger-\linebreak\phantom{x}\hspace{1.2pt}many, Phone:\hspace{1.2pt}\hspace{1.2pt}+49\hspace{1.2pt}202\hspace{1.2pt}/\hspace{1.2pt}439\hspace{1.2pt}-\hspace{1.2pt}2531, Fax:\hspace{1.2pt}\hspace{1.2pt}+49\hspace{1.2pt}202\hspace{1.2pt}/\hspace{1.2pt}439\hspace{1.2pt}-\hspace{1.2pt}3724 E-Mail: wegner@math.uni-wuppertal.de.\vspace{1.6pt}}

\hspace{-1000pt}\footnote{\hspace{0pt}$^{\text{a}}$\,Anna Goli\'{n}ska is supported by the National Science Centre (Poland) grant no. 2013/10/A/ST1/00091.\vspace{1.6pt}}

\begin{abstract} It is folklore that a power bounded operator on a sequentially complete locally convex space generates a uniformly continuous $\Cnull$-semigroup which is given by the corresponding power series representation. Recently, Doma\'{n}ski  asked if in this result the assumption of being power bounded can be relaxed. We employ conditions introduced by \.{Z}elazko to give a weaker but still sufficient condition for generation and apply our results to operators on classical function and sequence spaces.
\end{abstract}

\maketitle

\vspace{-12pt}
\section{Introduction}\label{INTRO}

Given a Banach space, then every linear and continuous operator from the space into itself generates a $\Cnull$-semigroup which is given by an exponential series representation. In the Banach space world of $\Cnull$-semigroups this case of a continuous generator is thus considered to be the trivial situation. The picture changes completely already in the considerably harmless appearing case of complete metrizable spaces. Based on a question of Conejero \cite{Conejero} several relations between continuity of the generator, uniform continuity of the semigroup and validity of a series representation have recently been revealed by Albanese, Bonet, Ricker \cite[Thm.\,3.3 and Prop.\,3.2]{ABR10} and Frerick, Jord\'{a}, Kalmes, Wengenroth \cite{FJKW}.

\smallskip

In the general case of a sequentially complete locally convex space $X$, it seems that the only result available so far is the generation theorem mentioned in the book \cite{Y} of Yosida, which states that a power bounded operator is always a generator. Here, $A\in L(X)$ is \textit{power bounded} if all its powers form an equicontinuous subset of $L(X)$, i.e., if for every continuous seminorm $p$ on $X$ there exists a continuous seminorm $q$ on $X$ and a constant $M>0$ such that the estimate $p(A^nx)\leqslant{}M q(x)$ holds for all $n\in\mathbb{N}_0$ and all $x\in X$. It is straight forward to generalize the above statement as follows.

\smallskip

\begin{thm}\label{THM-Y} Let $A\in L(X)$. Assume that there exists $\mu>0$ such that $\mu{}A$ is power bounded. Then $A$ generates a uniformly continuous $\Cnull$-semigroup $\T$ which is given by the formula $T(t)=\sum_{n=0}^{\infty}(tA)^n/n!$ for $t\geqslant0$ where the series converges absolutely with respect to the topology of uniform convergence on the bounded subsets of $X$.
\end{thm}

\smallskip

Inspired by Allan \cite[p.~400]{A} in this paper  operators with the property assumed in Theorem \ref{THM-Y} are said to be \textit{a-bounded}. Recently, Doma\'{n}ski \cite{D} asked, if the statement above can be improved in the sense of a condition weaker than a-boundedness which still assures the series representation or at least the generator property. Frerick, Jord\'{a}, Kalmes, Wengenroth \cite{FJKW} characterized generation for a special class of Fr\'{e}chet spaces by a condition closely related to the notions of \textit{m-topologizable} and \textit{topologizable} operators due to \.{Z}elazko \cite{Z1, Z2}.

\smallskip

In this paper we consider a quantitative version of topologizability which is weaker than m-topologizabili\-ty but implies the generator property and guarantees a series representation. Also for the case of an m-topologizable operator this result is new and improves Theorem \ref{THM-Y}. In combination with Bonet \cite[Ex.\,6]{JB} the latter shows that there exist complete, non-normable spaces on which every continuous operator is a generator---as in the case of a Banach space---although in general this is well-known to be not the case, see \cite[Ex.\,1]{FJKW}. We provide examples which show that our result applies to a class of operators which is strictly larger than those of the m-topologizable ones. Our counterexamples show that topologizability alone in general neither is necessary nor sufficient for generation. A variation of our main result, see Theorem \ref{THM-2}, suggests that it might be possible that an operator generates a $\Cnull$-semigroup but that the series representation is only valid on a 
finite time interval and fails for large times. It is open if such $\Cnull$-semigroups does really occur in nature.

\smallskip

For the theory of locally convex spaces we refer to Meise, Vogt \cite{MV} and Jarchow \cite{Jarchow}. For basic facts about semigroups on locally convex spaces we refer to Yosida \cite{Y} and K{\=o}mura \cite[Section 1]{Komura}.

\bigskip

\vspace{-3pt}
\section{Notation}\label{NOT}

\smallskip

For the whole article let $X$ be a sequentially complete locally convex space. We denote by $\cs(X)$ the system of all continuous seminorms on $X$, by $\mathcal{B}$ the collection of all bounded subsets of $X$ and by $L(X)$ the space of all linear and continuous maps from $X$ into itself. We write $L_b(X)$, if $L(X)$ is furnished with the topology of uniform convergence on the bounded subsets of $X$ given by the seminorms 
$$
q_B(S)=\sup_{x\in B}q(Sx)
$$
for $S\in L(X)$, $B\in\mathcal{B}$ and $q\in\cs(X)$. Under a $\Cnull$-semigroup $\T$ on $X$ we understand a family of maps $T(t)\in L(X)$ such that $T(0)=\id_X$, $T(t+s)=T(t)T(s)$ for $t$, $s\geqslant0$ and $\lim_{t\rightarrow t_0}T(t)x=T(t_0)x$ for $x\in X$ and $t_0\geqslant0$. $\T$ is said to be uniformly continuous if $T(\cdot)\colon[0,\infty)\rightarrow L_b(X)$ is continuous. The generator $A\colon D(A)\rightarrow X$ of a $\Cnull$-semigroup $\T$ is defined by 
$$
Ax=\lim_{t\searrow0}{\textstyle\frac{T(t)x-x}{t}} \; \text{ for } x\in D(A)=\{x\in X\:;\:\lim_{t\searrow0}{\textstyle\frac{T(t)x-x}{t}}\text{ exists}\}.
$$
The aim of this article is to identify conditions which guarantee that for a given $A\in L(X)$ there exists a $\Cnull$-semigroup $\T$ such that $A$ is its generator. In Section \ref{INTRO} we mentioned already a classical condition of this type. In the remainder we use the following three conditions which appeared in the literature in different contexts. The first condition below is the assumption of Theorem \ref{THM-Y}. In view of Allan \cite[p.~400]{A} we say that an operator is \textit{a-bounded}, if
\begin{equation}\label{A-BDD}
\exists\:\mu>0\;\forall\:p\in\cs(X)\:\exists\:q\in\cs(X)\;\forall\:n\in\mathbb{N}_0,\,x\in X\colon p(A^nx)\leqslant{}\mu^n q(x)
\end{equation}
holds; this prevents confusion with the notion of a bounded operator in the sense of \cite[p.~375]{MV}. The following two conditions are due to \.{Z}elazko \cite{Z1, Z2} and arise in view of \eqref{A-BDD} by allowing that $\mu$ is not constant but may depend on $p$, $q$ or $n$. We say that $A$ is \textit{m-topologizable} if
\begin{equation}\label{M-TOP}
\forall\:p\in\cs(X)\:\exists\:q\in\cs(X),\,\mu>0\;\forall\:n\in\mathbb{N}_0,\,x\in X\colon p(A^nx)\leqslant \mu^n q(x)
\end{equation}
holds and we say that $A$ is \textit{topologizable} if
\begin{equation}\label{TOP}
\forall\:p\in\cs(X)\:\exists\:q\in\cs(X)\;\forall\:n\in\mathbb{N}_0\;\exists\:\mu_n>0\;\forall\:x\in X\colon p(A^nx)\leqslant \mu_n\,q(x) 
\end{equation}
is valid. In both conditions $\mu$ depends on $p$, and in the sense that different configurations of $n$ and $q$ are possible, also on $q$. Clearly, every a-bounded operator is m-topologizable and every m-topologizable operator is topologizable. In Theorem \ref{THM-1} below we employ a quantitative version of topologizability; this will make clear why in our notation only the dependency $\mu=\mu_n$ is mentioned explicitly.

\bigskip

\vspace{-3pt}
\section{Generation}\label{SEC-0}

\smallskip

\begin{thm}\label{THM-1} Let $X$ be a sequentially complete locally convex space and $A\in L(X)$. Assume that\vspace{-5pt}
\begin{equation}\label{NEW-1}
\forall\:R>0,\,p\in\cs(X)\;\exists\:q\in\cs(X)\;\forall\:n\in\mathbb{N}_0\;\exists\:\mu_n>0\;\forall\:x\in X\colon p(A^nx)\leqslant\mu_n q(x)\text{ and }\!\!\Bigsum{n=0}{\infty}{\textstyle\frac{\mu_n}{n!}}R^n<\infty\vspace{-5pt}
\end{equation}
holds. Then $A$ generates a uniformly continuous semigroup $\T$ which is given by the exponential series expansion\vspace{-5pt}
\begin{equation*}
\text{(EXP)}\;\;\;\;\;\;T(t)=\Bigsum{n=0}{\infty}{\textstyle\frac{t^n}{n!}}A^n\vspace{-5pt}
\end{equation*}
where the latter converges in $L_b(X)$.
\end{thm}

\begin{proof} Fix $R>0$, $p\in\cs(X)$ and $B\subseteq X$ bounded. Select $q$ and $(\mu_n)_{n\in\mathbb{N}_0}$ such that \eqref{NEW-1} is satisfied. Put $K=\sup_{x\in B}q(x)$. By \eqref{NEW-1} we have $p(A^nx)\leqslant\mu_nq(x)$ and consequently $p(\frac{t^n}{n!}A^nx)\leqslant\frac{t^n}{n!}\mu_nq(x)$ for every $n\in\mathbb{N}_0$ and every $x\in X$. Thus for all $t<R$
$$
\Bigsum{n=0}{\infty}p_B({\textstyle\frac{t^n}{n!}}A^n)=\Bigsum{n=0}{\infty}\sup_{x\in B}p({\textstyle\frac{t^n}{n!}}A^nx)\leqslant \sup_{x\in B}q(x)\Bigsum{n=0}{\infty}{\textstyle\frac{t^n}{n!}}\mu_n=Kf(t)<\infty
$$
holds, where $f(t)=\Bigsum{n=0}{\infty}{\textstyle\frac{\mu_n}{n!}}t^n$ according to \eqref{NEW-1} defines a function $f\in C([0,R))$. This implies the absolute convergence of the exponential series in $L_b(X)$.

\smallskip

It is clear that $T(0)=\id_X$ holds. Let $t$, $s\geqslant0$ be given. Fix $x\in X$ and $p\in\cs(X)$. Select $q$ and $(\mu_n)_{n\in\mathbb{N}_0}$ according to \eqref{NEW-1} for some $R>\max\{s,t\}$. Put 
$$
a_i={\textstyle\frac{t^i}{i!}A^i},\;b_i={\textstyle\frac{s^i}{i!}A^i} \; \text{ and } \; c_i=\Bigsum{k=0}{i}a_kb_{i-k}.
$$
Denote by $A_n$, $B_n$ and $C_n$ the corresponding partial sums. By $A_{\infty}$ and $B_{\infty}$ we denote the limits of $A_n$ and $B_n$ in $L_b(X)$. Then
\begin{equation*}
C_n=\Bigsum{i=0}{n}\Bigsum{k=0}{i}a_kb_{i-k}=\Bigsum{i=0}{n}a_{n-i}B_i=A_nB_{\infty}+\Bigsum{i=0}{n}a_{n-i}(B_i-B_{\infty})
\end{equation*}
holds. We fix $\epsilon>0$. Since $A_n\rightarrow A_{\infty}$ in $L_s(X)$ there exists $L\in\mathbb{N}_0$ such that $p((A_n-A_{\infty})B_{\infty}x)\leqslant\frac{\epsilon}{3}$ holds for all $n\geqslant L$. Since $B_n\rightarrow B_{\infty}$ in $L_s(X)$ there exists $N\in\mathbb{N}_0$ such that $q((B_i-B_{\infty})x)\leqslant\frac{\epsilon}{3}\cdot f(t)^{-1}$ holds for $i\geqslant N$. Now we put $K=\max_{i=0,\dots,N-1}q((B_i-B_{\infty})x)$ and select $M\in\mathbb{N}_0$ such that the estimate $\sum_{i=M}^{\infty}\frac{t^i}{i!}\mu_i\leqslant\frac{\epsilon}{3}K^{-1}$ holds. For $n\geqslant\max\{L,M+N\}$ we compute
\begin{eqnarray*}
p((C_n-A_{\infty}B_{\infty})x) & = & p((A_n-A_{\infty})B_{\infty}x+\Bigsum{i=0}{n}a_{n-i}(B_i-B_{\infty})x)\\
& \leqslant & p((A_n-A_{\infty})B_{\infty}x)+\Bigsum{i=N}{n}p(a_{n-i}(B_i-B_{\infty})x)+\Bigsum{i=0}{N-1}p(a_{n-i}(B_i-B_{\infty})x)
\end{eqnarray*}
The first summand is less or equal to $\frac{\epsilon}{3}$ since $n\geqslant L$. For the second summand we estimate
\begin{equation*}
\Bigsum{i=N}{n}p(a_{n-i}(B_i-B_{\infty})x)\leqslant{}\Bigsum{i=N}{n}{\textstyle\frac{t^{n-i}}{(n-i)!}}\mu_{n-i}q((B_i-B_{\infty})x)\leqslant{}{\textstyle\frac{\epsilon}{3}}\cdot f(t)^{-1}\Bigsum{i=N}{n}{\textstyle\frac{t^{n-i}}{(n-i)!}}\mu_{n-i}\leqslant {\textstyle\frac{\epsilon}{3}}.
\end{equation*}
Here we used \eqref{NEW-1} as at the very beginning of this proof for the first estimate. The second follows from $i\geqslant N$ according to our selection of $N$. For the third summand we get
\begin{equation*}
\Bigsum{i=0}{N-1}p(a_{n-i}(B_i-B_{\infty})x)\leqslant\Bigsum{i=0}{N-1}{\textstyle\frac{t^{n-i}}{(n-i)!}}\mu_{n-i}q((B_i-B_{\infty})x)\leqslant K \Bigsum{i=0}{N-1}{\textstyle\frac{t^{n-i}}{(n-i)!}}\mu_{n-i}\leqslant K \Bigsum{i=M}{\infty}{\textstyle\frac{t^{i}}{i!}}\mu_{i}\leqslant{\textstyle\frac{\epsilon}{3}}.
\end{equation*}
This shows that $C_n\rightarrow A_{\infty}B_{\infty}$ holds in $L_s(X)$. By construction $A_{\infty}=T(t)$ and $B_{\infty}=T(s)$. Moreover, $C_n=\sum_{i=0}^n\frac{(t+s)^i}{i!}A^i$ by direct computation and thus by the first part $C_n\rightarrow T(t+s)$ in $L_s(X)$. This establishes the evolution property $T(t+s)=T(t)T(s)$.

\smallskip

Let $p\in\cs(X)$ and $B\subseteq X$ bounded be given. Select $q\in\cs(X)$ and $(\mu_n)_{n\in\mathbb{N}_0}$ as in \eqref{NEW-1} for $R=1$. Put $K=\sup_{x\in B}q(x)$ and denote by $f$ the function given by the power series in \eqref{NEW-1}. Then we have
\begin{equation*}
p_B(T(0)-T(t))=\sup_{x\in B}p\bigl(\Bigsum{n=1}{\infty}{\textstyle\frac{t^n}{n!}}A^nx\bigr)\leqslant K \Bigsum{n=1}{\infty}{\textstyle\frac{t^n}{n!}}\mu_n=K(f(t)-f(0))
\end{equation*}
for every $0\leqslant{}t<1$ which shows together with the evolution property that $T(\cdot)\colon[0,\infty)\rightarrow L_b(X)$ is continuous at $t=0$. We have the following condition
$$
\forall\:p\in\cs(X),\,R>0\;\exists\:q\in\cs(X),\,f\in C([0,R))\;\forall\:x\in X,\,t\in[0,R)\colon p(T(t)x)\leqslant f(t)q(x).
$$
Indeed, for given $p\in\cs(X)$, $R>0$ we select $q\in\cs(X)$ and $(\mu_n)_{n\in\mathbb{N}_0}$ as in \eqref{NEW-1}. We define $f\in C([0,R))$ by the series in \eqref{NEW-1} and use the estimate in \eqref{NEW-1} to obtain $p(T(t)x)\leqslant q(x)f(t)$ as desired. Let now $p\in\cs(X)$ and $B\subseteq X$ bounded be given. We fix $t>0$ and put $R=t+1$. Then we select $q\in\cs(X)$ and $f\in C([0,R))$ according to the condition above. For $0<h<1$ we compute
$$
p_B(T(t)-T(t+h)) = \sup_{x\in B}p(T(t)(T(0)-T(h))) \leqslant f(t)\sup_{x\in B}q(T(0)-T(h))=f(t)\,q_B(T(0)-T(h))
$$
which converges to zero for $h\searrow0$ by the first part. For $-t/2<h<0$ we compute
$$
p_B(T(t)-T(t+h)) = \sup_{x\in B}p(T(t+h)(T(-h)-T(0)))\leqslant f(t+h)\,q_B(T(-h)-T(0))
$$
which converges to zero for $h\nearrow0$ by the first part, since $f([0,t])\subseteq\mathbb{C}$ is bounded, and since $0\leqslant t+h \leqslant t$ holds for all $h$ under consideration. This establishes the continuity of $T(\cdot)\colon[0,\infty)\rightarrow L_b(X)$ at every $t>0$.
\smallskip
\\Finally let $p\in\cs(X)$ be given and select a last time $q\in\cs(X)$ and $(\mu_n)_{n\in\mathbb{N}_0}$ as in \eqref{NEW-1} for $R=1$. For $x\in X$ and $0<t<1$ we have
\begin{equation*}
p\bigl({\textstyle\frac{T(t)x-T(0)x}{t}}-Ax\bigr)=p\bigl({\textstyle\frac{1}{t}}\Bigsum{n=2}{\infty}{\textstyle\frac{t^n}{n!}}A^nx\bigr)\leqslant {\textstyle\frac{1}{t}}\,q(x) \Bigsum{n=2}{\infty}{\textstyle\frac{t^n}{n!}}\mu_n=q(x)\bigl({\textstyle\frac{f(t)-f(0)}{t}}-\mu_1\bigr)
\end{equation*}
which shows that the generator of $\T$ is indeed $A$.
\end{proof}

\begin{thm}\label{THM-2} Let $X$ be a sequentially complete locally convex space and $A\in L(X)$. Assume that\vspace{-5pt}
\begin{equation}\label{NEW-2}
\exists\:R>0\;\forall\:p\in\cs(X)\;\exists\:q\in\cs(X)\;\forall\:n\in\mathbb{N}_0\;\exists\:\mu_n>0\;\forall\:x\in X\colon p(A^nx)\leqslant\mu_n q(x)\text{ and }\!\!\Bigsum{n=0}{\infty}{\textstyle\frac{\mu_n}{n!}}R^n<\infty\vspace{-5pt}
\end{equation}
holds. Then $A$ generates a $\Cnull$-semigroup $\T$. The map $T(\cdot)\colon[0,R]\rightarrow L_b(X)$ is uniformly continuous and given by the exponential series expansion (EXP) where the latter converges in $L_b(X)$.
\end{thm}
\begin{proof} The proof of Theorem \ref{THM-1} shows that with $R>0$ as in \eqref{NEW-2} the formula (EXP) defines a function $T(\cdot)\colon[0,R]\rightarrow L(X)$ which is uniformly continuous and satisfies $T(s+t)=T(s)T(t)$ whenever $s$, $t\geqslant0$ are such that $s+t\leqslant R$ holds. For $t>R$ we define $T(t)=T(R)^nT(w)$ where $t=nR+w$ with $n\in\mathbb{N}_0$ and $0\leqslant w<R$ and get the $\Cnull$-semigroup $\T$ whose generator is $A$ by the proof of Theorem \ref{THM-1}.
\end{proof}

\begin{rem}\label{REM} 
\begin{compactitem}
\item[(i)] Minor adjustments in the proof of Theorem \ref{THM-1} show that if $A$ satisfies condition \eqref{NEW-1} or condition \eqref{NEW-2} then it even generates a $\Cnull$-group.\vspace{3pt}
\item[(ii)] Let $\Gamma\subseteq\cs(X)$ be a fundamental system of seminorms. Each of the conditions \eqref{A-BDD}-\eqref{NEW-2} holds already if the condition is true with $\Gamma$ instead of the set $\cs(X)$ of all continuous seminorms. For the conditions \eqref{M-TOP}-\eqref{NEW-2} the equivalence holds; \eqref{A-BDD} is satisfied if and only if
\begin{equation}\label{A-BDD-GEN}
\exists\:\mu>0\;\forall\:p\in\Gamma\:\exists\:q\in\Gamma,\,M\geqslant0\;\forall\:n\in\mathbb{N}_0,\,x\in X\colon p(A^nx)\leqslant{}M\mu^n q(x)
\end{equation}
is valid for some, or equivalently, for all fundamental systems of seminorms $\Gamma$.\vspace{3pt}
\item[(iii)] Under the assumptions of Theorem \ref{THM-2} the series representation (EXP) holds for $0\leqslant{}t\leqslant{}R$. The proof does not show that it also holds for $t>R$. However, there is no conrete example of a $\Cnull$-semigroup with this property.
\end{compactitem}
\end{rem}

\bigskip
\vspace{-10pt}
\section{Examples}

\smallskip

Let $B=(b_{j,k})_{j,k\in\mathbb{N}}$ be a K\"othe matrix \cite[Section 27]{MV}, i.e., $0\leqslant{}b_{j,k}\leqslant b_{j,k+1}$ holds for all $j$, $k\in\mathbb{N}$ and for every $j\in\mathbb{N}$ there exists $k\in\mathbb{N}$ such that $b_{j,k}>0$ holds. Let
\begin{eqnarray*}
\lambda^r(B)&=&\big\{x\in\mathbb{C}^\mathbb{N}\:;\:\forall\:k\in\mathbb{N}\colon\:\|x\|_k=\bigl(\Bigsum{j=1}{\infty}|b_{j,k}x_j|^r\bigr)^{1/r}<\infty\big\},\\
\lambda^{\infty}(B)&=&\big\{x\in\mathbb{C}^\mathbb{N}\:;\:\forall\:k\in\mathbb{N}\colon\:\|x\|_k=\sup_{j\in\mathbb{N}} b_{j,k}|x_j|<\infty\big\}
\end{eqnarray*}
denote the K\"othe echelon spaces of order $r\in[1,\infty]$, which are Fr\'{e}chet spaces with the fundamental system $\Gamma=(\|\cdot\|_q)_{q\in\mathbb{N}}$ of seminorms. In view of the conditions explained in Section \ref{NOT} and \ref{SEC-0} we identify $q\equiv\|\cdot\|_q$, cf.~also Remark \ref{REM}(ii).

\smallskip

Our first result shows that for diagonal operators $Ax=(a_jx_j)_{j\in\mathbb{N}}$ the situation is similar to the case in a classical Banach sequence space if and only if $(a_j)_{j\in\mathbb{N}}$ is bounded and that in this case Theorem \ref{THM-Y} is applicable.

\smallskip

\begin{prop}\label{PROP-1} Let $r\in[1,\infty]$ and let $B$ be a K\"othe matrix such that there exists a continuous norm on $\lambda^r(B)$. Let $(a_j)_{j\in\mathbb{N}}\subseteq\mathbb{C}$ be such that $A\colon\lambda^r(B)\rightarrow\lambda^r(B)$ with $Ax=(a_jx_j)_{j\in\mathbb{N}}$ is well-defined. Then the following are equivalent.
\begin{compactitem}\vspace{3pt}
\item[(i)] The operator $A$ is a-bounded.\vspace{2pt}
\item[(ii)] The operator $A$ is m-topologizable.\vspace{2pt}
\item[(iii)] The sequence $(a_j)_{j\in\mathbb{N}}$ is bounded.\vspace{5pt}
\end{compactitem}
If (i)--(iii) are satisfied then $A$ generates a $\Cnull$-semigroup by Theorem \ref{THM-Y}.
\end{prop}
\begin{proof} Let $1\leqslant{}r<\infty$. The case $r=\infty$ is similar. Since $\lambda^r(B)$ has a continuous norm, there exists $k_0\in\mathbb{N}$ such that $b_{j,k_0}>0$ holds for $j\in\mathbb{N}$. W.l.o.g.~we may assume $k_0=1$ and get that $b_{j,k}>0$ for all $k$ and $j\in\mathbb{N}$. Now we prove the equivalences.
\smallskip
\\(i)$\Rightarrow$(ii) This is true in general.
\smallskip
\\(ii)$\Rightarrow$(iii) Let $A$ satisfy \eqref{M-TOP} and assume that $(a_j)_{j\in\mathbb{N}}$ is unbounded. Let $p\in\mathbb{N}$ be given. Select $q\in\mathbb{N}$ and $\mu\geqslant0$ as in \eqref{M-TOP}. As $(a_j)_{j\in\mathbb{N}}$ is unbounded we can select $j_0\in\mathbb{N}$ with $|a_{j_0}|>\mu$. Put $x=(\delta_{j_0,j})_{j\in\mathbb{N}}$. By \eqref{M-TOP} we have
\begin{equation*}
\sup_{n\in\mathbb{N}_0}\|\mu{}^{-n}A^nx\|_{p}\leqslant\|x\|_{q}<\infty.
\end{equation*}
On the other hand we compute
\begin{equation*}
\sup_{n\in\mathbb{N}_0}\|\mu{}^{-n}A^nx\|_p
=\sup_{n\in\mathbb{N}_0}\mu{}^{-n}\bigl(\Bigsum{j=1}{\infty}|b_{j,q}a_j^nx_j|^r\bigr)^{1/r}
=\sup_{n\in\mathbb{N}_0}b_{j_0,q}\bigl|{\textstyle\frac{a_{j_0}}{\mu}}\bigr|^n
=\infty
\end{equation*}
which gives the desired contradiction.
\smallskip
\\(iii)$\Rightarrow$(i) Let $(a_j)_{j\in\mathbb{N}}$ be bounded. Put $\mu:=\sup_{j\in\mathbb{N}}|a_j|$. Let $p\equiv\|\cdot\|_p$ be given. Select $q=p$ and $M=1$. We estimate
\begin{equation*}
\|A^nx\|_p=\bigl(\Bigsum{j=1}{\infty}|b_{j,p}a_j^nx_j|^r\bigr)^{1/r} \leqslant \bigl(\Bigsum{j=1}{\infty}|b_{j,q}\lambda^nx_j|^r\bigr)^{1/r}\leqslant \mu^n\|x\|_q
\end{equation*}
for $x\in\lambda^{r}(B)$ and $n\in\mathbb{N}_0$ which shows \eqref{A-BDD-GEN}.
\end{proof}

\smallskip

Proposition \ref{PROP-1} applies in particular to diagonal operators on power series spaces \cite[Section 29]{MV}. Let us mention that many classical Fr\'{e}chet function spaces allow for a sequence space representation within this class of spaces.

\smallskip

For $r\in[1,\infty)$  Bonet, Ricker \cite[Prop.~5.5]{BR} showed that there exists an unbounded sequence $(\alpha_j)_{j\in\mathbb{N}}\subseteq\mathbb{C}$ such that $A\colon\lambda^r(B)\rightarrow\lambda^r(B)$, $Ax=(a_jx_j)_{j\in\mathbb{N}}$, is well-defined if and only if there exists an infinite subset $J\subseteq\mathbb{N}$ such that the sectional subspace $\lambda^r(J,B)=\bigl\{x\chi_J\:;\:x\in\lambda^r(J,B)\bigr\}$ see \cite[p.~481]{BR} for details, is Schwartz. Using \cite[Prop.~2.2]{BR} it follows that the latter is for instance the case if $\lambda^r(B)$ is Montel. This exhibits a large class of sequence spaces on which Theorem \ref{THM-Y} does not even apply to every diagonal operator.

\smallskip

Next, we give an example of an operator $A$ on a K\"othe echelon space  which is not a-bounded but m-topologizable and thus satisfies the assumptions of Theorem \ref{THM-1}. In particular this shows that the statements in Proposition \ref{PROP-1} are not equivalent for every K\"othe matrix.

\smallskip

\begin{ex}\label{EX-0} Let $B=(b_{j,k})_{j,k\in\mathbb{N}}$ be given by $b_{j,k}=1$ for $j\leqslant k$ and zero otherwise. Then, $\lambda^1(B)=\lambda^{\infty}(B)=\omega$ is the space of all sequences endowed with the topology of pointwise convergence. Let $A\colon \omega\rightarrow\omega$ be defined via $Ax=(jx_j)_{j\in\mathbb{N}}$ which gives rise to a linear and continuous operator. Then $A$ is not a-bounded, but m-topologizable and a generator by Theorem \ref{THM-1}. 
\end{ex}
\begin{proof} Firstly, we observe that $A\in L(X)$ holds since for each $k$ and each $x\in\omega$ the estimate $\|Ax\|_k\leqslant\|x\|_{k+1}$ is valid. Assume that $A$ is a-bounded. Select $\mu>0$ as in \eqref{A-BDD}. Take $p>\mu$ and put $p\equiv\|\cdot\|_p$. Select $q\equiv\|\cdot\|_q$ and $M\geqslant0$ as in \eqref{A-BDD}. Put $x=(1,1,\dots)$. By \eqref{A-BDD-GEN} the estimate $\|A^nx\|_p\leqslant M\mu^n\|x\|_m$ holds for every $n\in\mathbb{N}_0$ and thus $\sup_{n\in\mathbb{N}_0}\mu^{-n}\|A^nx\|_p<\infty$ holds. We compute
\begin{equation*}
\sup_{n\in\mathbb{N}_0}\mu^{-n}\|A^nx\|_p=\sup_{n\in\mathbb{N}_0}\mu^{-n}\max_{j=1,\dots,p}|j^n\cdot1|=\sup_{n\in\mathbb{N}_0}\,\bigl({\textstyle\frac{p}{\mu}}\bigr)^n=\infty
\end{equation*}
since $p/\mu>1$. Contradiction.
\smallskip
\\Now we show that $A$ is m-topologizable. Let $p\equiv\|\cdot\|_p$ be given. Select $q=p$ and $\mu=p>0$. Then for $n\in\mathbb{N}_0$ and $x\in\omega$ the estimate
\begin{equation*}
\|A^nx\|_p=\max_{j=1,\dots,p}|j^nx_j|\leqslant p^n \max_{j=1,\dots,p}|x_j|=\mu^n\|x\|_q
\end{equation*}
is true. Theorem \ref{THM-1} can be applied since \eqref{NEW-1} is satisfied with $(\mu_n)_{n\in\mathbb{N}_0}\equiv{}(\mu_{n,p,R})_{n\in\mathbb{N}_0}\equiv{}p$.
\end{proof}

\smallskip

A combination of Proposition \ref{PROP-1} and Example \ref{EX-0} even gives a characterization of those K\"othe spaces which have a continuous norm. In particular, Corollary \ref{PEPE-S-COR} shows that m-topologizable operators which are not a-bounded do exist on a large class of Fr\'{e}chet spaces.

\smallskip

\begin{cor}\label{PEPE-S-COR} Let $B$ be a K\"othe matrix and let $r\in[1,\infty]$ be fixed. Then $\lambda^r(B)$ has a continuous norm if and only if every m-topologizable diagonal operator on $\lambda^r(B)$ is a-bounded.
\end{cor}
\begin{proof} \textquotedblleft{}$\Rightarrow$\textquotedblright{} This is Proposition \ref{PROP-1}.

\smallskip

\textquotedblleft{}$\Leftarrow$\textquotedblright{} The following arises from an inspection of the proof of a classical result due to Bessage, Pe{\l}czy{\'n}ski \cite{BP}, see Bonet, Perez Carreras \cite[Thm.~2.6.13]{BPC}. Under the assumption that $\lambda^r(B)$ has no continuous norm, the latter proof shows that
$$
T\colon\omega\rightarrow\lambda^r(B),\,(a_n)_{n\in\mathbb{N}}\mapsto\Bigsum{n=1}{\infty}a_nx_n
$$
is an isomorphism whenever the sequence $(x_n)_{n\in\mathbb{N}}$ is selected such that $x_n\in\ker\|\cdot\|_n\backslash\ker\|\cdot\|_{n+1}$ for any $n\in\mathbb{N}$. Here, w.l.o.g.~we assume that $\ker\|\cdot\|_{n+1}\subset\ker\|\cdot\|_n$ is a strict subspace. Taking $\ker\|\cdot\|_n=\lambda^r(\{j\:;\:b_{n,j}=0\},B)$ into account it follows that we can select  $x_n$ to be $j_n$-th unit vector with $j_n$ such that $b_{n,j_n}=0$ and $b_{n+1,j_n}>0$. This selection produces a sequence $(x_n)_{n\in\mathbb{N}}$ of vectors with pairwise disjoint supports. With this selection the image of $T$ is exactly the sectional subspace $\lambda^r(J,B)$ with $J=\{j_n\:;\:n\in\mathbb{N}\}$. We define the operator $A\colon\lambda^r(B)\rightarrow\lambda^r(B)$ via $Ay=(a_jy_j)_{j\in\mathbb{N}}$ with
$$
a_j=\begin{cases}
\; jx_j & \text{ for } j\in J,\\
\hspace{6.5pt} 0& \text{ otherwise},
\end{cases}
$$
and conclude from Example \ref{EX-0} and the above that $A$ is not a-bounded but m-topologizable.
\end{proof}

\smallskip

We conclude this section with an example of an operator which is not m-topologizable but which nevertheless satisfies condition \eqref{NEW-1} so that Theorem \ref{THM-1} can still be applied.

\smallskip

\begin{ex} Let $B=(b_{j,k})_{j,k\in\mathbb{N}}$ be given by $b_{j,k}=j^k$. Then, $\lambda^1(B)=\lambda^{\infty}(B)=s$ is the space of rapidly decreasing sequences. Let $A\colon s\rightarrow s$ be defined via $Ax=(\log j\cdot x_j)_{j\in\mathbb{N}}$ which gives rise to a linear and continuous operator. Then the following statements are true.
\begin{compactitem}\vspace{3pt}
\item[(i)] Condition \eqref{NEW-1} holds and $A$ thus generates a uniformly continuous $\Cnull$-semigroup by Theorem \ref{THM-1}.\vspace{2pt}
\item[(ii)] In condition \eqref{NEW-1}, w.r.t.~$\Gamma=(\|\cdot\|_{q})_{q\in\mathbb{N}}$, the constants $\mu_n>0$ depend necessarily on $R>0$.\vspace{2pt}
\item[(iii)] The operator $A$ is topologizable but not m-topologizable.
\end{compactitem}
\end{ex}
\begin{proof} For $p\in\mathbb{N}$ and $x\in s$ we compute
$$
\|Ax\|_p= \sup_{j\in\mathbb{N}}\,\log j\abs{x_j}j^p=\sup_{j\in\mathbb{N}}j^{-1}\log j\abs{x_j}j^{p+1} \leqslant e^{-1}\|x\|_{p+1}
$$
which shows $A\in L(s)$. 

\smallskip

(i) Let $p\in\mathbb{N}$ and $R>0$ be given. Select $q\in\mathbb{N}$ such that $q-p>R$. For given $n\in\mathbb{N}$ put $\mu_n=(n/(q-p))^ne^{-n}$.  For $x\in s$ we have
$$
\|A^nx\|_p= \sup_{j\in\mathbb{N}}\,(\log j)^n|x_j|j^p=\sup_{j\in\mathbb{N}}\,(\log j)^n j^{-(q-p)}\norm{x}_q\leq \bigl({\textstyle\frac{n}{q-p}}\bigr)^ne^{-n}\norm{x}_q=\mu_n\|x\|_q
$$
and by the Cauchy root test the series
$$
\Bigsum{n=0}{\infty}{\textstyle\frac{\mu_n}{n!}}R^n=\Bigsum{n=0}{\infty}{\textstyle\frac{n^n\,R^n}{e^n\,(q-p)\,n!}}
$$
converges.

\smallskip

(ii) Let $p\in\mathbb{N}$. We assume that we can select $q\in\mathbb{N}$ and $(\mu_n)_{n\in\mathbb{N}_0}$ such that $\|A^nx\|_p\leqslant\mu_n\|x\|_q$ holds for every $x\in s$ and such that $\Bigsum{n=0}{\infty}{\textstyle\frac{\mu_n}{n!}}R^n$ is convergent for every $R>0$. For $j\in\mathbb{N}$ we denote by $e_j$ the $j$-th unit vector and compute
$$
\|A^ne_j\|_p=(\log j)^n j^p=(\log j)^n j^{p-q}\|e_j\|_q
$$
which yields
$$
\mu_n\geqslant\sup_{j\in\mathbb{N}}\,(\log j)^n j^{p-q}.
$$
Computations show that for every fixed $n\in\mathbb{N}_0$ with $(q-p)|n$ the supremum over $j\in\mathbb{N}$ is attained for $j=e^{n/(q-p)}$ and is equal to $(n/(q-p))^ne^{-n}$. We get
$$
\limsup_{n\rightarrow\infty} \bigl({\textstyle\frac{\mu_n}{n!}}\bigr)^{1/n}\geqslant\hspace{-8pt}\lim_{\stackrel{\scriptstyle\phantom{R} n\rightarrow\infty\phantom{R}}{\scriptstyle\phantom{R}(q-p)|n\phantom{R}}}\bigl({\textstyle\frac{\mu_n}{n!}}\bigr)^{1/n}\geqslant\hspace{-8pt}\lim_{\stackrel{\scriptstyle\phantom{R} n\rightarrow\infty\phantom{R}}{\scriptstyle\phantom{R}(q-p)|n\phantom{R}}}\bigl(n^n/((q-p)^n e^{n}n!)\bigr)^{1/n}={\textstyle\frac{1}{q-p}}
$$
that is the radius of convergence of the power series above is less or equal to $q-p$. Contradiction.

\smallskip

(iii) Condition \eqref{NEW-1} implies that $A$ is topologizable. Assume that $A$ is m-topologizable. Let $p\in\mathbb{N}$ be given. We select $q\in\mathbb{N}$ and $\mu>0$ according to \eqref{M-TOP}. Then for $n\in\mathbb{N}_0$ and $x\in X$ the estimate $\|A^nx\|_p\leqslant\mu^n\|x\|_q$ holds. The estimates and computations performed in the proof of (ii) show that 
$$
\mu^n\geqslant \bigl({\textstyle\frac{n}{q-p}}\bigr)^{n}e^{-n}
$$
holds for infinitely many $n\in\mathbb{N}_0$ which is not possible with a finite $\mu>0$. Contradiction.
\end{proof}

\bigskip

\vspace{-10pt}
\section{Counterexamples}

\smallskip

In this final section we show that topologizability alone, i.e., without any growth control on the $\mu_{n}$ is neither sufficient nor necessary for the generation of a $\Cnull$-semigroup. To establish our first example we need the following lemma which is kind of a counterpart of \cite[II.2.3]{EN}.

\smallskip

\begin{lem} \label{l-restr-sem} Let $X$, $Y$ be Fr\'echet spaces such that $Y\subseteq X$ holds with continuous inclusion map. Let $A\in L(X)$ be such that $B:=A|_Y\in L(Y)$ and $B$ generates a $\Cnull$-group $(S(t))_{t\in\mathbb{R}}$. If $A$ generates a $\Cnull$-semigroup $\T$, then $T(t)\vert_Y=S(t)$ is valid for all $t\geqslant0$. 
\end{lem}

\begin{proof} Take $y\in Y$, $t>0$, $\delta>0$ and define the function ${\varphi}\colon {[0,t]}\rightarrow {X}$, $\varphi(s):=T(s)S(t-s)y$. We show that $\varphi$ is differentiable and that we have $\varphi'(s)=0$ for all $s$. Fix $s\in[0,t]$ and take $(h_n)_{n\in\N}\subseteq\mathbb{R}$, $h_n\rightarrow 0$. Then
\begin{align*}
(\varphi(s)-\varphi(s+h_n))/h_n
&= h_n^{-1}\bigl(T(s)S(t-s)y-T(s+h_n)S(t-s-h_n)y\bigr)\\
&=T(s)\bigl[h_n^{-1}\bigl(S(t-s)y-S(t-s-h_n)y\bigr)\\
&\phantom{=\;}+h_n^{-1}\bigl(S(t-s-h_n)y-T(h_n)S(t-s-h_n)y\bigr)\bigr].
\end{align*}
For $n\rightarrow\infty$ the first summand in the bracket on the right hand side converges to $BS(t-s)y$ in $Y$ and hence in $X$. For the second summand we define the sequence $(A_n)_{n\in\mathbb{N}}\in L(X)$, $A_nx=h_n^{-1}(\id_X-T(h_n))x$ for $x\in X$, which converges pointwise to $-A$. By the Banach-Steinhaus Theorem \cite[11.1.3]{Jarchow} we get that $A_n\rightarrow -A$ converges uniformly on the precompact subsets of $X$. The set $\{S(u)y\:;\:u\in [-\delta, t]\}$ is compact in $Y$, hence precompact in $X$. Thus for every seminorm $p\in\cs(X)$ and every $\epsilon>0$ there exists $N_1$ such that for all $n\geqslant N_1$ the estimate $p((A_n+A)S(t-s-h_n)y)\leqslant\epsilon/2$ holds. Denote by $\iota\colon Y\hookrightarrow X$ the embedding map. As $A\circ\iota$ is continuous, there exist a seminorm $q\in\cs(Y)$ and a constant $c>0$ such that $p(A(S(t-s)-S(t-s-h_n))y)\leqslant c q( (S(t-s)-S(t-s-h_n))y)$ for all $n\in\mathbb{N}$. Since $(S(t))_{t\in\mathbb{R}}$ is strongly continuous, there exists $N_2$ such 
that $q((S(t-s)-S(t-s-h_n))y)<\epsilon/(2c)$ for all $n\geqslant N_2$. Put $N=\max(N_1,N_2)$. Then for $n\geqslant N$ we have
$$
p\bigl(A_nS(t-s-h_n)y+AS(t-s)y\bigr)\leqslant
 p\bigl( (A_n+A)S(t-s-h_n)y\bigr)+p\bigl(A\bigl(S(t-s)-S(t-s-h_n)\bigr)y\bigr)\leqslant\epsilon.
$$
This shows that the second summand converges to $-AS(t-s)y$. Since $T(s)$ is continuous and both $(h_n)_{n\in\mathbb{N}}$ and $s$ were arbitrary we get that
$$
\varphi'(s)=T(s)\bigl(AS(t-s)y-BS(t-s)y\bigr)=0.
$$
Hence $S(t)y=\varphi(0)=\varphi(t)=T(t)y$. Since $y$ and $t$ were arbitrary this finishes the proof.
\end{proof}

\smallskip

Now we are prepared to show that the differentiation operator on $H(\mathbb{D})$ does not generate a $\Cnull$-semigroup. It was shown in \cite[Ex.\,3]{JB} that the latter is topologizable. For the convenience of the reader we repeat the proof for the latter fact.

\smallskip

\begin{ex} Let $H(\mathbb{D})$ be the space of holomorphic functions on the unit disc. The differentiation operator $A\in L(H(\mathbb{D}))$, $Af=f'$, is topologizable but does not generate a $C_0$-semigroup.
\end{ex}
\begin{proof} Firstly we show that $A$ is topologizable, cf.~\cite[Ex.\,3]{JB}. The topology of $H(\mathbb{D})$ is given by the seminorms $\norm{f}_q=\sup_{\abs{z}\leqslant q}\abs{f(z)}$, $0<q<1$. For each $n\in\N$, by Cauchy integral formula we get
$$
\norm{A^nf}_q=\sup_{\abs{z}\leqslant{}q}\abs{f^{n}(z)}\leqslant\sup_{\abs{z}\leqslant{}q}\,\bigl|{\textstyle\frac{n!}{2\pi i}}\int_{\abs{w=s}}{\textstyle\frac{f(w)}{(w-z)^{n+1}}}\bigr|dw\leqslant{\textstyle\frac{n!s}{(s-q)^{n+1}}}\norm{f}_s
$$
for any $q<s<1$ and $f\in H(\mathbb{D})$. 

\smallskip

To show that $A$ does not generate a $\Cnull$-semigroup we use Lemma \ref{l-restr-sem}: Let $X=H(\mathbb{D})$ and $Y=H(\mathbb{C})$ be the space of all holomorphic functions on the complex plane with the topology given by seminorms $\norm{f}_q=\sup_{\abs{z}\leqslant{}q}\abs{f(z)}$, $0<q<\infty$. We identify $f\in H(\mathbb{C})$ with its restriction $f|_{\mathbb{D}}\in H(\mathbb{D})$ and thus have $Y=H(\mathbb{C})\subseteq H(\mathbb{D})=X$. As before, the Cauchy integral formula gives $\norm{A^nf}_q=\|f^{(n)}\|_q\leqslant\frac{n!s}{(s-q)^{n+1}}\norm{f}_s$ for any $q<s$, $n\in\mathbb{N}_0$ and $f\in H(\mathbb{C})$. Hence, for arbitrary $0<q<\infty$ and $R>0$ we take $s>R+q$, put $\mu_n=\frac{n!s}{(s-q)^n}$ and get that $A$ satisfies the assumptions of Theorem \ref{THM-1}. Thus,
$A$ generates a strongly continuous group on $H(\mathbb{C})$, cf.~Remark \ref{REM}(ii). However, $A$ does not extend to $H(\mathbb{D})$. Indeed, take the sequence $f_k(z)=\sum_{n=0}^kx^n$ converging to $f(z)=\frac{1}{1-z}$ in $H(\mathbb{D})$ and consider
$$
\lim_{k\rightarrow \infty} T(1)f_k\bigl({\textstyle\frac{3}{4}}\bigr)=\lim_{k\rightarrow\infty} \Bigsum{n=0}{k}{\textstyle\frac{n!}{k!(n-k)!}}\bigl({\textstyle\frac{3}{4}}\bigr)^{k-n} \geqslant\lim_{k\rightarrow\infty} \bigl({\textstyle\frac{3}{4}}\bigr)^{k}\Bigsum{n=0}{k}{\textstyle{k \choose n}}= \lim_{k\rightarrow\infty}\bigl({\textstyle\frac{3}{4}}\bigr)^{k}2^{k}=\infty.
$$
Hence, Lemma \ref{l-restr-sem} implies that $A$ does not generate a $C_0$-semigroup on $H(\mathbb{D})$.
\end{proof}

\smallskip

Finally we show that the differentiation operator in $C^{\infty}(\mathbb{R})$ is not topologizable. It is well known that the latter generates the shift semigroup, cf.~\cite[Ex.~1]{FJKW}.

\smallskip

\begin{ex} The differentiation operator $Af=f'$ generates a $\Cnull$-semigroup on the space  $C^\infty(\mathbb{R})$ of smooth functions but is not topologizable.
\end{ex}
\begin{proof} The topology of $C^\infty(\mathbb{R})$ is given by the seminorms $\norm{f}_{K,p}=\sup_{x\in K, \alpha\leqslant p}\abs{f^{\alpha}(x)}$ for $K\subseteq\mathbb{R}$ compact and $p\in\N$.
It is well-known that the translation semigroup $\T$ defined by $T(t)f(x)=f(t+x)$ for $f\in C^\infty(\mathbb{R})$ and $t\geqslant0$ is strongly continuous and that it is generated by $A$.

\smallskip

Assume that $A$ is topologizable. Fix the seminorm $\norm{\cdot}_{K,p}$ with $K=[0,2\pi]$ and $p\in\mathbb{N}$. Then there exists a seminorm $\norm{\cdot}_{L,q}$ and a sequence $(\mu_n)_{n\in\N_0}\subseteq (0,\infty)$ such that $\|A^nf\|_{K,p}\leq \mu_n \norm{f}_{L,q}$ holds for every $n\in\N_0$ and $f\in C^\infty(\mathbb{R})$. Put $f_k(x)=\sin(xk)$, $k\in\N$ and $n=q-p+1$. Then for all $k\in\mathbb{N}$ we have $k^{n+p+1}=\norm{A^nf_k}_{[0,2\pi],p}\leqslant\mu_n\norm{f}_{L,q}=\mu_n k^q$. Contradiction. 
\end{proof}

\smallskip

\footnotesize

{\sc Acknowledgements. }The authors would like to thank J.~Bonet for pointing out the statement of Corollary \ref{PEPE-S-COR} to them. In addition they would like to thank L.~Frerick, E.~Jord\'{a} and J.~Wengenroth for valuable comments in the context of the conditions used in \cite{FJKW} and \cite{Z1, Z2}.

\normalsize

\bibliographystyle{amsplain}

\bigskip


\providecommand{\href}[2]{#2}

\end{document}